\documentclass[11pt]{amsart}
\usepackage{amsmath}
\usepackage{amssymb,amscd}
\usepackage{graphicx}
\usepackage[colorlinks=true, urlcolor=blue,bookmarks=true,bookmarksopen=true, citecolor=blue,hypertex]{hyperref}

\numberwithin{equation}{section}

\newtheorem{defn}{Definition}[section]
\newtheorem{theorem}{Theorem}[section]
\newtheorem{lemma}{Lemma}[section]

\newtheorem{corollary}[theorem]{Corollary}

\theoremstyle{definition}
\newtheorem{example}[theorem]{Example}

\newtheorem{remark}[theorem]{REMARK}

\def \begineq{\begin{equation}}
\def \endeq{\end{equation}}

\def \bb{\mathbb}

\def \CC{{\bb{C}}}
\def \CP{{\bb{CP}}}

\def \RR{{\bb{R}}}
\def \TT{{\bb{T}}}

\def \ZZ{{\bb{Z}}}

\def \({\left(}
\def \){\right)}
\def \<{\langle}
\def \>{\rangle}
\def \bar{\overline}

\begin{document}

\title{ A class of torus manifolds with nonconvex orbit space}
 %MSC classification
 %53C15; 53D20

\author[M. Poddar]{Mainak Poddar}

\address{Departamento de
Matem\'aticas, Universidad de los Andes, Bogota, Colombia}

\email{mainakp@gmail.com}

\author[S. Sarkar]{Soumen Sarkar}

\address{Department of Mathematics, Korea Advanced Institute of Science
and Technology, Daejeon, Republic of Korea}

\email{soumensarkar20@gmail.com}



\thanks{}

\abstract  We study a class of smooth torus manifolds whose orbit
space has the combinatorial structure of a simple polytope with
holes. We construct moment angle manifolds for such polytopes with
holes and use them to prove that the associated torus manifolds
admit stable almost complex structure. We give a combinatorial
formula for the Hirzebruch $\chi_y$ genus of these torus
manifolds. We show that they have (invariant) almost complex
structure if they admit positive omniorientation. We give examples
of almost complex manifolds that do not admit a complex structure.
When the dimension is four, we calculate the homology groups and
describe a method for computing the cohomology ring.

\endabstract

\maketitle

\section{Introduction}\label{intro}

The moment polytope of the Hamiltonian action of the real torus
$\TT^n$ on a smooth projective toric variety (toric manifold) may
be identified with the orbit space of the action. The moment
 polytope (Delzant polytope) is rather rigid with severe integrality
constraints, see Definition 2.1.1 of \cite{Sil}. In 1991 Davis and
Januskiewicz \cite{DJ} introduced a generalization of toric
manifolds, now known as quasitoric manifolds, which may be
obtained as identification spaces of $ \TT^n \times  P$ where $P$
is a simple $n$-dimensional polytope. In general these spaces do
not have algebraic or invariant symplectic structure, but they
still have a lot of remarkable properties; see the survey
\cite{BP}. In this article we study a class of even dimensional
manifolds which may be obtained as identification space of $\TT^n
\times P$ where $P$ is not convex, but a simple polytope with
holes which are also simple polytopes. In \cite{Mas} and
\cite{HM}, Masuda and Hattori introduced the notion of torus
manifold (see Definition \ref{def:tm}). Our manifolds are a
special class of torus manifolds. As in the case of quasitoric
manifolds, the torus action on our manifolds is {\it locally
standard}, i.e. locally equivalent to the natural action, up to
automorphism, of $U(1)^n$ on $\CC^n$.

 We describe the combinatorial construction of these manifolds in
 section \ref{def}. However,
these manifolds are also obtained by gluing quasitoric manifolds
along deleted neighborhoods of principal torus orbits (Lemma
\ref{smooth}). We refer to this as the fiber sum construction. It is a
special case of a more general construction in \cite{GK}.
 This is used to endow the manifolds with smooth structure (Lemma
\ref{smooth}), and in certain cases with almost complex structure
(Theorem \ref{acs}).

 We realize each of our manifolds as the quotient of a submanifold of $\CC^m$
  by the free action of a compact torus in section \ref{mac}.
  This may be viewed
 as a topological analogue of the construction of toric manifolds by
 symplectic reduction, or of the quotient construction of toric varieties.
  We use this to endow our
  manifolds with a stable complex structure (Lemma \ref{stable}).
Using it, we give a combinatorial formula for the $\chi_y$ genus
of these manifolds (Theorem \ref{chiy}) following the work of
Panov \cite{Pan} in quasitoric case. The formula also follows from
Lemma \ref{stable} and  a more general result in \cite{HM}.

 Our manifolds admit almost
complex structure if they admit a positive omniorientation (Lemma
\ref{eqiacs} and Theorem \ref{acs}). Positive omniorientation is
also a necessary condition if we require the almost complex
structure to be $\TT^n$-invariant.

 These
 manifolds cannot admit an invariant symplectic structure (Lemma \ref{symp1}) or
 an invariant integrable complex structure (see \cite{IK}) if
the orbit space has at least one hole. It would be interesting to
know if any of these torus manifolds admit a
 symplectic or complex structure.
 If the orbit space has
one hole, then the manifold can not be Kahler (Lemma
\ref{kahler}). We give examples of almost complex manifolds that
do not admit a complex structure in section \ref{iq}.

A lot is known about the topological invariants of these manifolds
from the works \cite{Mas} and \cite{HM}. However as they  have
nontrivial homology in odd degrees (see Theorem 9.3 of \cite{MP}),
the formula for the cohomology ring given in Corollary 7.8 of
\cite{MP} does not hold when the orbit space has holes. Even
explicit formulas for their (co)homology groups are not known in
general. In section \ref{homology}, we give a combinatorial
formula for the homology groups when dimension is four. We also
describe a method for computing the cohomology ring for the four
dimensional manifolds.

\section{Construction and smooth structure}\label{def}

\subsection{Polytope with holes}
A polytope is the convex hull of a finite set of points in
$\RR^n$. An $n$-dimensional polytope is said to be simple if every
vertex is the intersection of exactly $n$ codimension one faces.
 Let $P_0$ be an $n$-dimensional simple  polytope
in $\RR^n$. Let $ P_1, P_2, \ldots, P_s $ be a disjoint collection
of simple polytopes belonging to the interior of $P_0$. Let
 \begin{equation}\label{holes}
 P = P_0 - \bigcup_{k=1}^{s} P_k^{\circ}.
 \end{equation}
 We call $P$ an
$n$-dimensional $polytope ~ with ~ simple ~holes$. The polytopes
$P_1, P_2, \ldots , P_s$ are called holes of $P$.  The faces of $P$
are the faces of $ P_k, \, k = 0, \ldots, s $.
%In the following
%figure \ref{figeg01} the shaded regions represent polytopes with
%simple holes.
%
\begin{figure}[ht]
        \centerline{
           \scalebox{0.60}{
            \input{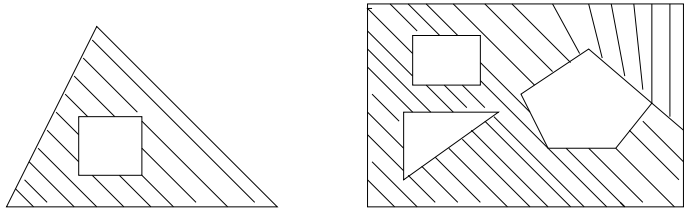}
            }
          }
       \caption {Polytopes with simple holes in $\RR^2$.}
       \label{figeg01}
      \end{figure}

\subsection{Combinatorial construction}
 Let $P$ be an $n$-dimensional simple  polytope with $s$ simple
holes. Let $ \mathcal{F}(P) = \{ F_1, F_2,\ldots, F_m \} $ be the
set of all codimension one faces ({\em facets}) of $P$. Note that
$\mathcal{F}(P) = \bigcup_{k=0}^{s} \mathcal{F}(P_k)$. Also, if
$F$ is a nonempty face of $P$ of codimension $k$ then $F$ is the
intersection of a unique collection of $k$ facets of $P$.
 The following definition is a
straightforward generalization of the notion of characteristic
function for a simple polytope, which is a crucial concept for
studying quasitoric manifolds \cite{DJ,BP}.

\begin{defn} A function
$ \lambda \colon \mathcal{F}(P) \rightarrow \ZZ^n $  is called a
 characteristic function if it satisfies the following
condition: Whenever $F = \bigcap_{j=1}^{k} F_{i_j} $ is a
$(n-k)$-dimensional face of $P$, the span of the vectors $ \lambda
(F_{i_1}), \lambda (F_{i_2}),\ldots, \lambda(F_{i_k}) $ is a
$k$-dimensional direct summand of $\ZZ^n$. We will denote
$\lambda(F_i)$ by $\lambda_i$ for simplicity and call it the
characteristic vector of $F_i$.
\end{defn}

For any face $F = \bigcap_{j=1}^{k} F_{i_j}$ of $P$, let $N(F)$ be
the submodule of $\ZZ^{n}$ generated by  $ \lambda_{i_1}, \ldots,
\lambda_{i_k} $.
 The module $ N(F) $ defines a sub-torus $G_F$ of $\TT^n = \ZZ^n
\otimes \RR /\ZZ^n =\RR^n/\ZZ^n $ as follows.
\begin{equation} G_F := (N(F) \otimes \RR) /N(F).
\end{equation}

Define an equivalence relation $ \sim $ on the product
space $ \TT^n \times P $ by
\begin{equation}\label{ereln}
 (t,x) \sim (u,y) ~\mbox{if}~  x=y ~\mbox{and} ~ u^{-1}t \in
 G_{F(x)}
\end{equation}
where $F(x)$ is the unique face of $P$ whose relative interior
contains $x$.

We denote the quotient space as follows.
 \begin{equation}\label{quo}
  M = M(P,\lambda) := (\TT^n \times P)/\sim.
\end{equation}

 The space $M$ is a $2n$-dimensional manifold. The proof of this
 is analogous to the quasitoric case \cite{DJ}.
The  $\TT^n$ action on $( \TT^n \times P )$ induces a natural
effective action of $\TT^n$ on $M$, which is locally standard (see
\cite{DJ}). Let $ \pi \colon M\rightarrow P $  be the projection
or orbit map defined by $ \pi ([ (t,x) ]) = x $.

\begin{defn}\label{def:tm}\cite{HM}  A closed, connected, oriented, smooth manifold $Y$ of
dimension $2n$ with an effective smooth action of $\TT^n$ with
non-empty fixed point set is called a torus manifold if a
preferred orientation is given for each characteristic
submanifold. A characteristic submanifold is, by definition, any
codimension two closed connected submanifold of $Y$, which is
fixed by some circle subgroup of $\TT^n$ and contains at least one
$\TT^n$-fixed point.
\end{defn}

In the case of $M$, the $\TT^n$-fixed point set corresponds
bijectively to the set of vertices of $P$. Observe that the spaces
$X_i := \pi^{-1}(F_i)$,  $i = 1, \ldots, m, $ are the
characteristic submanifolds of $M$. Each $X_i$ is a
$2(n-1)$-dimensional quasitoric manifold. We explain in section
\ref{omni}, how the characteristic function $\lambda$ endows each
$X_i$ with a preferred orientation.
 We say that $M$ is the torus manifold derived from
  the {\em characteristic pair} $(P, \lambda)$.

\subsection{Fiber sum construction}

\begin{lemma}\label{smooth}
The torus manifold $M(P, \lambda)$ is smooth and orientable.
\end{lemma}

\begin{proof}
By induction it is sufficient to prove that $M(P, \lambda)$ has a smooth structure
  when $P$ is a polytope with one hole, that is,
 $P = P_0 - P_1^0$. Let $\mathcal{F}(P_0)$ and $\mathcal{F}(P_1)$
be the set of facets of $P_0$ and $P_1$ respectively.  The
restrictions $\lambda_0$ and $\lambda_1$ of $\lambda$ on
$\mathcal{F}(P_0) $ and $ \mathcal{F}(P_1)$ are characteristic
functions on $P_0$ and $P_1$ respectively. Let $M_0$ and $M_1$ be
the quasitoric manifolds associated to the characteristic pairs
$(P_0, \lambda_0)$ and $(P_1, \lambda_1)$ respectively. These
manifolds, being quasitoric, have smooth structure.
% (An explicit proof of this is given in \cite{Pod}
%following ideas in \cite{BR}.)

Let $\pi_k : M_k \to P_k, \; k=0,1$  be the orbit maps. Fix
points $x_k \in P_k^{\circ}$. Let
\begin{equation} L_k = \pi_k^{-1}(x_k).
\end{equation}

Let $U_k \subset M_k$ be a $\TT^n$ invariant neighborhood of $L_k$
such that
 \begin{equation}\label{bk} B_k :=\pi_k(U_k) \subset P_k
\end{equation}
  is diffeomorphic to an open ball in $\RR^n$.

The quasitoric manifolds $M_k$ are orientable. An orientation on
$M_k$ is determined by orientations on $\RR^n \supset P_k$ and
$\TT^{n}$. Suppose ${\bf p}= (p_1,\ldots, p_n)$ and ${\bf q}=
(q_1, \ldots, q_n)$ be the standard Cartesian and angular
coordinates on $\RR^n$ and $\TT^n$ respectively. Then the
orientation on $M_k$ corresponding to the ordering
$\frac{\partial}{\partial p_1}, \frac{\partial}{\partial q_1},
\ldots, \frac{\partial}{\partial p_n}, \frac{\partial}{\partial
q_n} $ will be assumed.
 %We fix an orientation on $\TT^{n}$ for once and for
%all, corresponding to the standard orientation on its Lie algebra.
%We also induce orientations on each $P_k$ from the standard
%orientation on $\RR^{n}$.

By \eqref{bk} there exist equivariant orientation
preserving diffeomorphisms
\begin{equation}\label{fk} f_k : U_k  \to \TT^n \times B,
\end{equation}
where $B$ is the unit $n$-ball centered at the origin.
 Denote the punctured unit $n$-ball, $B-\{0\}$, by $B^{-}$.
 %Let \begin{equation}B_k^{-} := \pi_k(U_k - L_k)= B_k - \{ x_k\}. \end{equation}
 Let
$|\cdot|$ be the Euclidean norm on $\RR^n$. Define
\begin{equation} r:= |{\bf p}| \quad {\rm and} \quad { \Theta}=(
\theta_1, \ldots, \theta_{n}) := \frac{{\bf p}}{r}.
\end{equation}

 The space $M(P, \lambda)$ can be obtained from $M_0 -
L_0$ and $M_1- L_1$ by identifying $U_0-L_0$ and $U_1 - L_1$ as
follows.
Let $ g : B^{-} \to B^{-}$
 be the orientation preserving involution,
\begin{equation}\label{g} g({\bf p}) = \frac{{1 -r}}{r} (p_1, \ldots,
p_{n-1}, -p_n).
\end{equation}
In other words, $  g(r, \Theta) = (1 -r, \theta_1,
\ldots, \theta_{n-1}, -\theta_{n} ) $.

 Define
\begin{equation}\label{h}
 h =f_0^{-1} \circ (Id \times g) \circ f_1.
\end{equation}
Identify $U_0-L_0$ with $U_1 - L_1$ by the orientation preserving
equivariant diffeomorphism $h$.
\end{proof}

\begin{remark}\label{orbitsp} Note that the map $g$ as used in \eqref{h}, radially inverts
a deleted neighborhood $B_1^{-}$ of the point $x_1$ in
$P_1^{\circ}$ and reflects it about the hyperplane $p_n=0$. Then
the map $h$ identifies it to a deleted neighborhood of $x_0$ in
$P_0^{\circ}$. Up to homeomorphism, we can widen the puncture at
$x_0$, and fit $P_1-B_1$ into it and thus recover our picture of
the orbit space of the glued manifold as the polytope with hole
$P_0 - P_1^{\circ}$. A smooth embedding of this orbit space into
Euclidean space is described in section \ref{mac}. We do not know
for sure if the smooth structure on the orbit space coincides with
the smooth structure of $P$ coming from its given embedding in
$\RR^n$.
\end{remark}

%
% $ g$ is orientation reversing since the Jacobian
% $dg =  r^{-3}  (1- r^2 )^{-1/2}  (r^2(1-r^2 ) I_n - PP^t)$
% where $P$ is the column vector with entries $(p_1, \ldots, p_n)$.
% The matrix $PP^t$ has eigenvalues $r^2$ and $0$ with algebraic
% multiplicity $1$ and $n-1$ respectively. Since the characteristic polynomial
% of $PP^t$ goes to $\infty$ and has nonzero slope at $r^2$, hence det(dg)
% must be negative. Note $0 < r < 1$. For the above calculation it is good to
% use $dr = \sum \frac{p_i}{r} dp_i$, and
% $d(\frac{p_i}{r} \sqrt{1-r^2}  ) = \frac{\sqrt{1-r^2}}{r} dp_i - \frac{p_i}{r^2\sqrt{1-r^2}} dr$
%
\begin{remark}\label{fibersum} We refer to the above gluing
construction as fiber sum construction because of its similarity
to the symplectic fiber sum construction (see \cite{Gro}, \cite{Gom}). A more general fiber sum
construction for spaces with torus action was introduced in
\cite{GK}.
\end{remark}

\begin{remark}\label{reasonforg}
 As we may observe from section \ref{mac}, the exact formula for the gluing map $g$ is
not important for the smooth structure.
\end{remark}

\begin{remark}\label{signofchar} The sign of the characteristic
vectors do not affect the equivariant diffeomorphism type of $M$.
This follows from similar observation for quasitoric manifolds,
see \cite{DJ, BR}.
\end{remark}

\subsection{Omniorientation}\label{omni}
We fix an orientation for $M(P, \lambda)$ as above by choosing
standard orientations on $\TT^n$ and $\RR^n$. Also each
characteristic submanifold $X_i$ is quasitoric and hence
orientable.

\begin{defn} An omniorientation  is an assignment of orientation for
$M(P, \lambda)$ as well as for each $X_i$. Given such an
assignment, we say that $M(P, \lambda)$ is omnioriented.
\end{defn}

Given the above choice of orientation for $M$, the characteristic
function $\lambda$  determines a natural omniorientation on $M$ as
follows: The  characteristic vector $\lambda_i$ determines a
fiberwise $S^1$ action on the normal bundle of $X_i$,
corresponding to the isotropy group $G_{F_i}$. This equips the
normal bundle with a complex structure and therefore an
orientation. This, together with the orientation on $M$, induces
an orientation on $X_i$. We will refer to this omniorientation as
the {\it characteristic omniorientation}.

Consider an omniorientation on $M$. Let $v \in M$ be a fixed point
of the $\TT^n$ action (or corresponding vertex of $P$). If the
orientation of $T_v(M)$ determined by the orientation on $M$ and
the orientations of characteristic submanifolds containing $v$
coincide then the {\it sign} $\sigma(v)$ is defined to be $1$,
otherwise $\sigma(v)$ is $-1$.
\begin{defn} An omniorientation is called positive if $\sigma(v) = 1$
for each fixed point $v$.\end{defn}

For the characteristic omniorientation, the sign of a vertex $v$
may be computed as follows \cite{BP}. Suppose $v= F_{i_1} \cap
\ldots \cap F_{i_n}$. To each codimension one face $F_{i_k}$
assign the unique edge $E_k$ such that $E_k \cap F_{i_k} = v$. Let
$e_k$ be a vector along $E_k$ with origin at $v$. Order (rename)
the $e_k$s so that $e_1, \ldots, e_n$ is a positively oriented
basis for $\RR^n$. Consider the corresponding matrix
$\Lambda_{(v)}= [\lambda_{i_1} \ldots \lambda_{i_n} ]$. Then
\begin{equation}\label{sign}
\sigma(v) = \det \Lambda_{(v)}.
\end{equation}

\begin{remark}\label{int}
It is also evident that the oriented intersection number of the
submanifolds $X_{i_1}, \ldots, X_{i_n}$ is $\sigma(v)$.
\end{remark}

\section{Calculations in dimension four}\label{homology}

 Let $\pi: M(P, \lambda) \to P$ be a
$4$-dimensional torus manifold, where $P$ is a polytope with $s$
simple holes. We give a $CW$ structure on $M(P, \lambda)$ and
compute the homology groups.

First assume that $P$ has only one hole. Then $P= P_0 - P_1^0$,
where $P_0$ and $P_1$ are simple $2$-dimensional  polytopes with
vertices $\{v_1, \ldots, v_{l_0}\}$ and $\{ u_1, \ldots,
u_{l_1}\}$ respectively.
 Assume
that $dist(v_1u_1) \leq dist(v_1u_j)$ for all $j= 1, \ldots, l_1$.
 Let $E_{v_i}$  and $E_{u_j}$
be the edges of $P$ joining the vertices $\{v_{i}, v_{i+1} \}$ and
$\{u_{j}, u_{j+1}\}$ respectively for $i= 1, \ldots, l_0;\, j= 1,
\ldots, l_1$. Here assume $v_{l_0+1} =v_1$ and $u_{l_1+1}=u_1$.
  Let $E_{v_1 u_1}$ be the line segment joining
$v_1$ and $u_1$.

We construct the $i$-th skeleton $X_i$ of $M(P, \lambda)$ as
follows. Let $X_0= \{v_1, \ldots, v_{l_0-1}, u_1, \ldots,$ $
u_{l_1}\}$.
 Define \begin{equation} \begin{array}{ll}
e^1_i=(\{(1,1)\}\times E_{v_i})/\sim & {\rm for} \; i=1, \ldots,
l_0-2\\
  e^1_{l_0-1}= (\{(1,1)\} \times E_{v_1u_1})/\sim & \\
e^1_{l_0+j-1} = (\{(1,1)\} \times E_{u_j})/\sim  & {\rm for}\;
j=1,
\ldots, l_1 \\
 X_1= \cup_{i=1}^{l_0 + l_1 -1} \bar{e^1_i}. & \\
\end{array}
\end{equation}
A picture of the $1$-skeleton for a polytope with one hole is
given in figure \ref{ega5} (a) on the next page. Define
\begin{equation} \begin{array}{ll}
 e^2_i= ((\TT^2 \times
E_{v_i})/\sim) - \bar{e^1_i} & {\rm for}\; i=1, \ldots, l_0-2  \\
 e^2_{l_0-1}= (( \{1\} \times S^1 \times E_{v_1u_1})/\sim )-
\bar{e^1_{l_0-1}} & \\
 e^2_{l_0}= ((S^1 \times \{1\} \times E_{v_1u_1})/\sim) -
 \bar{e^1_{l_0-1}} & \\
 e^2_{l_0+j} = ((\TT^2 \times E_{u_j})/\sim) - \bar{e^1_{l_0+j-1}} &
 {\rm for}\; j=1,  \ldots, l_1 \\
  X_2 = \cup_{i=1}^{l_0 + l_1} \bar{e^2_i}. & \\
\end{array}
\end{equation}
Define \begin{equation} \begin{array}{l} e^3 = ((\TT^2 \times
E_{v_1u_1})/\sim) - (\bar{e^2_{l_0 -1}} \cup \bar{e^2_{l_0}} )\\
  X_3 = \bar{e^3} \cup X_2. \end{array} \end{equation}
Define
\begin{equation}\label{u4} U^4 = P -\{E_{v_1} \cup \ldots \cup E_{v_{l_0-2}} \cup
\partial P_1 \cup E_{v_1u_1}\}.\end{equation}
 Clearly $U^4$ is homeomorphic to
$\RR_{\geq 0}^2$. So \begin{equation}(\TT^2 \times U^4)/\sim ~
\cong B^4 = \{x\in \RR^4 : |x| < 1\}.\end{equation}
 Define \begin{equation}
e^4 = (\TT^2 \times U^4)/\sim \; {\rm and} \; X_4 = \bar{e^4}
\end{equation}

For the above $CW$ structure, by reasons of either dimension or
orientation, the cellular boundary maps $d_2, d_3, d_4$  are zero.
Since $X_1$ is homotopic to a circle, we get the following result.
\begin{theorem}\label{homt4} Suppose $P$ is a $2$-polytope with
one hole. Then
$$H_i(M(P, \lambda), \ZZ) = \left\{ \begin{array}{ll}
 \displaystyle  \ZZ^{l_0 + l_1} & \mbox{if} ~ i = 2\\
 \ZZ & \mbox{if}~ i = 0, 1, 3, 4 \\
 0 & \mbox{if} ~ i > 4.
\end{array} \right.$$
\end{theorem}
We can give a similar $CW$ structure on $M(P, \lambda)$ when $P$
is a $2$-polytope with multiple holes.  The figure \ref{ega5} (b)
gives a representation of the $1$-skeleton of such a structure
when there are two holes.
\begin{figure}[ht]
        \centerline{
           \scalebox{0.70}{
            \input{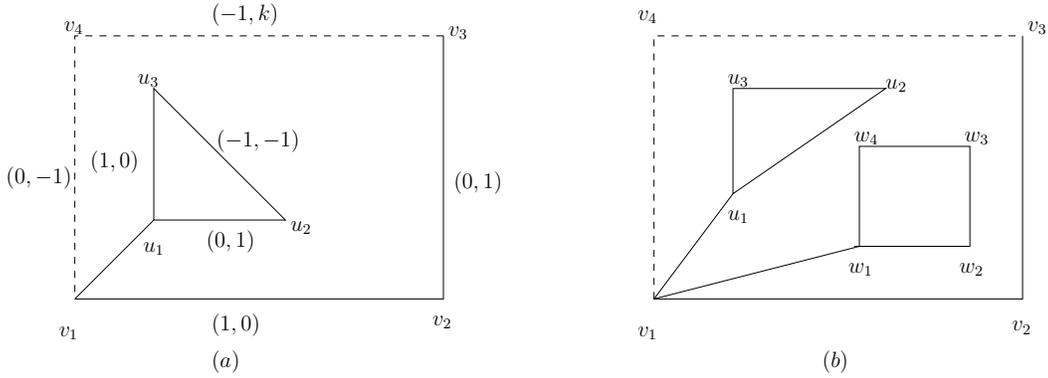}
            }
          }
       \caption {$1$-skeleta for $2$-polytopes with (a) one hole and (b) two holes.}
       \label{ega5}
     \end{figure}

\begin{corollary}\label{homt5}
 Suppose $P$ is a $2$-polytope with $m$ vertices and
$s$ simple holes. Then $$H_i(M(P, \lambda), \ZZ) = \left\{
\begin{array}{ll} \displaystyle
 \ZZ^{ m + 2s -2} ~
 & \mbox{if} ~ i = 2\\
 \displaystyle  \ZZ^s & \mbox{if}~ i = 1, 3,\\
\ZZ & \mbox{if}~ i = 0, 4\\
 0 & \mbox{if} ~ i > 4.
\end{array} \right.$$
\end{corollary}

\subsection{Cohomology ring}\label{cohring}
Assume that $M$ has the characteristic omniorientation. In
dimension four it is possible to compute the cohomology ring by
using Poincar\'e duality and intersection product. To illustrate,
we consider the case  when there is one hole. Let $x_k \in H_2(M)$
denote the homology class of the sphere associated to the $2$-cell
$e^2_{k}$. Here characteristic orientation is chosen for the
sphere if $k \neq l_0 -1, l_0$. Otherwise orientation determined
by the direction $v_1 u_1$ and standard orientation of the
associated $S^1$ is assumed.

The products of two classes $x_i$ and $x_j$ when $i$ and $j$ are
both less than $l_0-1$, is the same as obtained by considering
them as classes in $H_{\ast}(M_0)$. This is because the homotopies
needed to achieve transversality can be done away from a
neighborhood of any given principal torus fiber. Similar remarks
apply when $i$ and $j$ both exceed $l_0$. If $i < l_0-1$ and $j >
l_0$, or vice versa, then the product is obviously zero.

 \begin{figure}[ht]
        \centerline{
           \scalebox{0.80}{
            \input{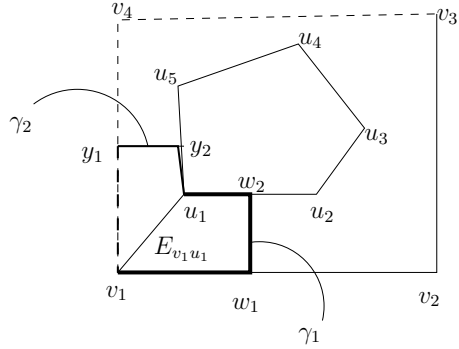}
            }
          }
       \caption {Homotopic copies of $x_{l_0-1}$, here $l_0 =5$ and $l_1=4$.}
       \label{figeg02}
      \end{figure}

 Now
consider the class $x_{l_0-1}$. To compute the self intersection $
x_{l_0-1}^2$, we choose two different homotopy representatives,
$S_1^2$ and $S_2^2$, of $x_{l_0-1}$ which intersect only at $v_1$
and $u_1$. Let $w_1, y_1$ be points in the relative interior of
the edges $v_1v_2$ and $v_{l_0}v_1$ respectively. Similarly let
$w_2, y_2$ be points in the relative interior of the edges
$u_1u_2$ and $u_{l_1}u_1$ respectively. Let $\gamma_1, \;
\gamma_2$ be the piecewise linear paths $v_1 w_1 w_2 u_1$ and
$v_1y_1y_2u_1$ respectively. Let $S_i^2$ be the homotopy sphere $
(\{1\} \times S^1 ) \times \gamma_i /\sim$. The circle subgroup $
\{1\} \times S^1 $ corresponds to the submodule of $\ZZ^2$
generated by $(0,1)$.
   It is possible to express $(0,1)$
uniquely as an integral linear combination $a_1 \lambda_{1} + a_2
\lambda_{l_0} $. Let $d =  \det [\lambda_{l_0}, \lambda_{1} ] =
\sigma(v_1) $. Near $v_1$, the sphere $S_1^2$ is homotopic to $a_2
d $ times the characteristic sphere over $v_1v_2$. Similarly
$S_2^2$ is homotopic to $ a_1 d$ times the characteristic sphere
over $v_1v_{l_0}$.
 Therefore the contribution of $v_1$ to $ x_{l_0-1}^2$ is $ (d)(a_1 d)
(a_2 d) = a_1 a_2 d$, see remark \ref{int}. The contribution from
the point $u_1$ may be calculated similarly. Other intersection
products of degree $2$ classes may be calculated by using similar
homotopies. For example, $x_1 \cdot x_{l_0-1} = (d)(a_1 d) =a_1 $.
Finally the intersection of the generators degree one and degree
three homology classes is $1$ up to sign.

\begin{example}
Consider $M$ to be the fiber sum of a Hirzebruch surface with
$\CP^2$ corresponding to figure \ref{ega5}(a).
 Let $x_1, \ldots, x_7$ be the generators of $H_2(M)$ as defined above.
  Let $y$ and $z$ be the
generators of $H_1(M)$ and $H_3(M)$ corresponding to the cells
$\sum_{j=4}^6 e^1_{j}$ and $e^3$ respectively. Then,
\begin{equation} \begin{array}{l}
 x_1^2= x_3^2= x_4^2=0, x_2^2=-k,  x_i^2= 1 \; {\rm if} \; i \ge 5, \\
 x_1x_3= x_2x_3 =x_2x_4=x_3x_6= x_3 x_7= x_4x_5= x_4x_6= 0, \\
x_ix_j=0\; {\rm if}\; i=1,2 \; {\rm and}\; j= 5, 6, 7, \\
x_1 x_2= x_1x_4= x_5 x_6= x_5x_7= x_6x_7 =1, x_3x_5 = x_4x_7= -1,
yz=1.
\end{array}
\end{equation}
\end{example}

\section{Moment angle manifold}\label{mac}

Let $M$ be the manifold obtained by fiber summing the smooth
quasitoric manifolds $M_i(P_i, \lambda_i)$. We may assume that
each $P_i$ lies in a distinct copy of $\RR^n$. Let $S_i$ be the
one point compactification of the copy of $\RR^n$ that contains
$P_i$, with standard smooth structure.

 The orbit space $O$ of $M$ inherits a smooth structure from the
gluing operations in Lemma \ref{smooth}. As noted in Remark
\ref{orbitsp}, $O$ is homeomorphic to $P$. Using the punctured
balls $B_k^{-}$ as tubes between different affine copies of
$\RR^n$, we may construct a smooth embedding of $O$ into
$\RR^{n+s}$, where $s$ is the number of holes of $P$. However, we
need more. Consider the manifold with corners $O^+$,  obtained by
gluing a punctured copy of each $S_i$, $1\le i \le s$, to $P_0$
punctured at $s$ points, according to the gluing maps in Lemma
\ref{smooth}.  Then $O^+$ is homeomorphic to $P_0$. We may
smoothly embed $O^+$ into $\RR^{n+s}$.

For notational simplicity we will describe the embedding in terms
of $P$ and the $P_i$'s. Induce smooth structures on $P$ and $P_0$
using the homeomorphisms with $O$ and $O^+$ respectively.
 Then there exists a smooth embedding $\psi_0$ of $P_0$ in
$\RR^{n+s}=\{(p_1, \ldots, p_{n+s}) \}$
 such that the following hold:
\begin{enumerate}
 \item The image of $P- \bigcup_{k=1}^s V_k$, where $V_k$ is a small neighborhood of $P_k$ in $\RR^n$,
   lies in $\RR^n =\{p_1, \ldots, p_n\}$.
 \item The image of $V_k$ lies in the $(n+1)$-dimensional subspace
$\{p_{n+j}=0 | 1\le j \neq k \le s \}$. The  projection of
$\psi_0(V_k)$ to $\RR^n$ lies inside $V_k$.
 \item The embedding $\psi_0$ is affine when restricted to the boundary $\partial P_k$.  The image of
 $\partial P_k$ lies in the affine subspace $H_k :=\{ p_{n+k}=1,  p_{n+j}=0\, \forall \,j \, {\rm such\; that}\, 1 \le j \neq k \le s\} $.
\item  $\psi_0(P_0) \bigcap H_K = \partial P_k$. \item The image
of $P$ lies between the affine subspaces $p_{n+k}=0$ and
$p_{n+k}=1$ for
 each  $1\le k \le s$.
\end{enumerate}

 Consider any facet $F_i$ of $P$. Suppose $F_i \subset P_k$ where $k
\ge 1$.  Choose a linear polynomial $A_i$ in the variables $p_1,
\ldots, p_n, p_{n+k} $, other than $ a_k :=1 -p_{n+k}$, which is
zero on $\psi_0(F_i)$ and positive on $ \psi_0(P) -\psi_0(F_i)$.
Define $d_i =A_i + a_k + \sum_{1 \le j\neq k} p_{n+j} $. If $F_i $
is a facet of $P_0$, then let $A_i$ be the defining linear
polynomial of $F_i$ in the variables $p_1, \ldots, p_n$ such that
$A_i$ is positive in the interior of $P_0$. In this case define
$d_i = A_i + \sum_{1\le j} p_{n+j}$.

 Then for a point $x$ in $\psi_0(P)$,
 $d_i(x)$ can be thought of as an $l_1$-distance of $x$ from the affine subspace of $\psi_0(F_i)$.
  We construct a smooth embedding $\psi_1$ of $\psi_0(P_0)$
into  $\RR^m$ by $\psi_1(x) = (d_1(x), \ldots, d_m(x))$ where $m=
|\mathcal{F}(P)|$. The composition $ \psi:= \psi_1 \circ \psi_0$
defines an embedding of $P_0$ into $\RR^m=\{(r_1, \ldots, r_m)\} $
such that the image of $P$ lies in $\RR^m_{\ge}= \{r_i \ge 0 \,
\forall \, i \}$.  Suppose $y \in \psi(P)$. Then $r_i(y)=0$ if and
only if $y \in \psi(F_i)$.

The space $\CC^m$ can be regarded as a quotient of $\TT^m \times
\RR^m_{\ge}$ by an equivalence relation $\sim_0$ as follows: Let
$u_1, \ldots, u_m$ denote the standard basis of $\ZZ^m$. Let $T_i$
denote the circle subgroup $(\ZZ u_i \otimes \RR)/ \ZZ u_i$ of
$\TT^m$.  For any face $F=\{r_j=0| j\in J \}$ of $\RR^m_{\ge}$, we
define the subgroup $T_F:= \prod_{j\in J} T_j$. For any $y$ in
$\RR^m_{\ge}$, let $F(y)$ denote the unique face of $\RR^m_{\ge}$
whose  relative interior contains $y$. Then define $\sim_0$ by
\begin{equation}\label{sim0}
(t,x) \sim_0 (u,y) ~\mbox{if}~  x=y ~\mbox{and} ~ u^{-1}t \in
 T_{F(y)}.
\end{equation}

\begin{defn} Let $\pi_0: \CC^m \to \RR^m$ denote the quotient map. Define
the moment angle complex $Z(P)$ of $P$ by $$ Z(P)=
\pi_0^{-1}(\psi(P)).$$
\end{defn}

 We may identify $\pi_0$ with the smooth map defined coordinate-wise
  by $z_i \mapsto |z_i|^2$. This shows that $Z(P)$ is smooth. The details
  are straightforward and left to the reader.

 Given a characteristic function $\lambda$ for $P$, let
$\Lambda: \ZZ^m \to \ZZ^n $ be the linear map defined by
$\Lambda(u_i) = \lambda_i$. Let $K= \ker \Lambda$ and $T_K =(K
\otimes \RR)/ K $. Then it is easy to observe that topologically
$Z(P)$ is a principal $T_K$ bundle over $M(P,\lambda)$.

 The leaf space $\mathcal{M}(P, \lambda)$ of the foliation corresponding
 to the
smooth and free action of $T_K$ on $Z(P)$ has a natural smooth
structure.  Since $\TT^m \cong T_K \times \TT^n$, it is not hard
to check that $\mathcal{M}(P, \lambda)$ supports a smooth action
of $\TT^n$. Moreover $\mathcal{M}(P, \lambda)$ is equivariantly
homeomorphic to $M(P,\lambda)$ with respect to this action.
 There is a one-to-one
correspondence between normal orbit types and, in fact, an
isomorphism of $\TT^n$-normal systems (see \cite{Dav}) of
$\mathcal{M}(P, \lambda)$ and $M(P, \lambda)$. (Here the smooth
structures on the orbit spaces match that of $O$.) All of these
may be ascertained by studying the local representations of
$Z(P)$, up to equivariant diffeomorphism, by $T_K \times \CC^k
\times (\CC^{\ast})^{n-k}  $ near the faces of $P$. Therefore by
Theorem 4.3 of \cite{Dav}, $\mathcal{M}(P, \lambda)$ and $M(P,
\lambda)$ are equivariantly diffeomorphic. We will henceforth
identify $M(P, \lambda)$ with $\mathcal{M}(P, \lambda)$ without
additional comments.

\section{Almost complex structure}

In this section we prove three results: i) That every
omniorientation of $M$ determines a stable almost complex
structure on it, ii) that
 if $M$  admits a positive omniorientation and $\dim(M)= 4$, then
 there exists an almost complex structure on $M$ which
  is equivalent to the associated stable complex structure,
   and iii) that
there exists a $\TT^{n}$-invariant almost complex structure on $M$
if and only if $M$ has a positive omniorientation. It is not known
to us if the invariant almost complex structure is equivalent to
the associated stable almost complex structure.

\begin{lemma}\label{stable} Every omniorientation of the torus
manifold $M(P, \lambda)$ determines a stable almost complex
structure on it.
\end{lemma}

\begin{proof} Let $Q= \psi(P)$ and $Q_k = \psi(P_k)$ where $\psi$ is the
embedding of $P_0$ into $\RR^m$ constructed in section \ref{mac}.
the normal bundle of $Q_0$ in $\RR^m$ is trivial as $Q_0$ is
contractible. Therefore the normal bundle $N_Q$ of $Q$ is also
trivial. We may in fact identify $N_Q$ with a tubular neighborhood
of $Q$ in $\RR_{\ge}^m$ following an idea in \cite{BR}: Identify
$N_Q$ with $\{(x,v)| x \in Q, \, v \in (T_xQ)^{\perp} \subset
\RR^m \}$ where $\perp$ denotes orthogonal complement with respect
to the dot product. Then define the map $f: N_Q \to \RR_{\ge}^m$
by $f(x,v)=(e^{v_1}x_1, \ldots, e^{v_m}x_m) $ where the $x_i$s and
$v_i$s denote the coordinates of $x$ and $v$ respectively. Then a
careful analysis of the situation shows that $v \cdot Df_{(x,0)}
(v) = \sum_{i=1}^m v_i^2 x_i $ is positive. This shows that
$Df_x(N_Q)$ is transversal to $T_xQ$. We identify $N_Q$ with
$Df(N_Q)$.

 Since a tubular neighborhood of $Q$ in $\RR^m_{\ge}$ pulls back to
 a tubular neighborhood of $Z(P)$ in $\CC^m$ under $\pi_0$, we may
 identify  the normal bundle $N_Z$ of $Z(P)$ in $\CC^m$ with $\pi_0^{\ast} N_Q $.
 Therefore $N_Z$ is trivial. Let $N_M$ denote the
pullback of $N_Q$ to $M$ under $\psi \circ \pi$.
  Then by a slight generalization of the Atiyah sequence \cite{At},
   we obtain the following split
exact sequence of bundles,
\begin{equation}
0 \rightarrow \mathfrak{t}_K \times M \rightarrow (TZ(P)\oplus
N_Z)^{T_K} \rightarrow TM \oplus N_M \rightarrow 0.
\end{equation}
Here $\mathfrak{t}_K$ denotes the Lie algebra of $T_K$. Since the
action of $T_K$ on $T\CC^m$ is complex linear, therefore
$(TZ(P)\oplus N_Z)^{T_K} = (T\CC^m|_{Z(P)})^{T_K} $ inherits a
complex structure. It follows that $M$ admits a stable almost
complex structure.
\end{proof}

%Let $N_i$ be the normal bundle to the characteristic submanifold
%$X_i = \pi^{-1}(F_i)$. Suppose $F_i$ is a codimension one face of
%$P_j$. Then there exists a complex line bundle $\nu_i$ on the
%quasitoric manifold $M_j$ such that $\nu_i |_{X_i} = N_i  $ and
%$\nu_i$ is trivial away from $X_i$, see \cite{BR} or \cite{DJ}.
%The complex structure on $\nu_i$ agrees with the one on $N_i$
%determined by the omniorientation. We extend $\nu_i$ trivially to
%a complex line bundle on $M$, which we denote by the same symbol.

%Any point in $M$ has a neighborhood $U$ which can be thought of as
%belonging to one of the quasitoric manifolds $M_j$ that glue to
%produce $M$. By a well-known result on stable complex structures
%on quasitoric manifolds, $TU$ is stably isomorphic to $\bigoplus_k
%\nu_{j_k}|_U $ where the line bundles $\nu_{j_k}$s correspond to
%the characteristic submanifolds of $M_j$. We may assume that the
%rest of the $\nu_i$s are trivial on $U$. Therefore
%$\bigoplus_{i=i}^m \nu_i|_U \cong TU \oplus (U \times
%\RR^{2m-2n})$. Consequently we have,
%\begin{equation} \bigoplus_{i=i}^m \nu_i \cong TM \oplus (M \times \RR^{2m-2n}).
%\end{equation}
%Since the bundle $\bigoplus_{i=i}^m \nu_i$ is complex, the proof
%is complete.

As $T_K$ acts diagonally on $\CC^m$, the bundle $
(T\CC^m|_{Z(P)})^{T_K}$ splits naturally into a direct sum of $m$
complex line bundles over $M$, namely $\nu_1, \ldots, \nu_m$,
corresponding to the complex coordinate directions of $\CC^m$.
These directions correspond to (distance from) the facets of $P$.
Since the angular direction $u_i$ maps to $\lambda_i$ by
$\Lambda$, the bundle $\nu_i$ restricts to the normal bundle of
$X_i$ on the characteristic submanifold $X_i$. The total Chern
class of $M(P, \lambda)$ associated to the above stable complex
structure admits the following product decomposition,
\begin{equation}\label{chern}
c(TM)= \prod_{i=1}^m (1+ c_1(\nu_i)).
\end{equation}

Using standard localization formula or Theorem \ref{chiy}, we
obtain
\begin{equation}\label{cn}
c_n(TM) = \sum \sigma(v)
\end{equation} where the sum is over all vertices of $P$.

\begin{lemma}\label{eqiacs}
If $M(P,\lambda)$ admits a positive orientation and $\dim(M)=4$
then it admits an almost complex structure which is equivalent to
the associated stable almost complex structure.
\end{lemma}

\begin{proof} By Theorem 1.7 of \cite{Th}, the lemma holds if
 $c_2(TM)= e(TM)$. This follows from
  \eqref{cn} and Corollary \ref{homt5}.
\end{proof}

\begin{theorem}\label{acs}
The torus manifold $M(P, \lambda)$ admits a $\TT^n$-invariant
almost complex structure if and only if it has a positive
omniorientation.
\end{theorem}

\begin{proof} The necessity of positive omniorientation for
existence of $\TT^n$-invariant almost complex structure follows
from similar argument as in quasitoric case, see \cite{BP}.

 To prove sufficiency, first assume that the number of holes is one.
 Note that a positive omniorientation of $M(P,\lambda)$ induces
 positive omniorientation  on $M_0$ and $M_1$. Then by the work of
  Kustarev \cite{Kus}, there exist $\TT^n$-invariant orthogonal almost
complex structures $J_k $ on $M_k$, $ k=0,1$. In particular, these
structures are orientation preserving. We may assume that the
complex structure $J_k$ is locally constant in the normal
direction near $L_k$, as explained below.

 Recall the
orientation preserving
 diffeomorphisms $f_k$ in \eqref{fk}. Since
$T( \TT^n  \times B)$ is trivial, $df_k$ defines an isomorphism
\begin{equation} df_k : TU_k  \to  \TT^n \times B \times \RR^{2n}.
\end{equation}
Consider the almost complex structures \begin{equation}
\widehat{J}_k = df_k \circ J_k \circ df_k^{-1}
\end{equation} on $  \TT^n \times B \times \RR^{2n}$.
 Choose a smooth non-decreasing function
  $\gamma: \RR \to \RR $ such that
\begin{equation} \gamma(t) = \left\{ \begin{array}{ll}
  0 & {\rm if} \; t \leq  \epsilon_1 \\
  t & {\rm if} \; t \geq   \epsilon_2  \\
                             \end{array} \right.
\end{equation} where $0 < \epsilon_1 <\epsilon_2 < 1 $ are small
real numbers. Define \begin{equation}
 {J}_k^{\prime} ({\bf q}, r, \Theta ) =
                             \widehat{J}_k ({\bf q}, \gamma(r),
                              \Theta ).
                              \end{equation}
 Replace $J_k$ by $df_k^{-1}  {J}_k^{\prime} df_k $ on $U_k$.
 Denote the resulting almost complex structure on $M_k$ by $J_k$
 without confusion. Note that these new almost complex structures
 are orientation preserving and $\TT^n$-invariant.

Recall the orientation preserving diffeomorphism
  $g$ in \eqref{g}. Define
\begin{equation}\label{phi}
\phi_0 := f_0 , \quad \phi_1 :=  ( Id \times g) \circ f_1 :
U_1-L_1 \to   \TT^n  \times B^{-}.
\end{equation}
 We have orientation preserving isomorphisms,
\begin{equation} d\phi_k : T(U_k -L_k) \to \TT^n  \times B^{-}
\times \RR^{2n}.
\end{equation}
Consider the almost complex structures \begin{equation}
\widetilde{J}_k = d\phi_k \circ J_k \circ d\phi_k^{-1}
\end{equation} on $\TT^n  \times B^{-} \times \RR^{2n}$.
The space of orientation preserving almost complex structures on
$\RR^{2n}$ may be identified with $GL^{+}(2n,\RR)/GL(n,\CC)$.
 Since $\phi_k$ is
orientation preserving, we can regard $\widetilde{J}_k$ as a map
\begin{equation} \widetilde{J}_k : \TT^n  \times B^{-} \to GL^{+}(2n,\RR)/GL(n,\CC).
\end{equation}
Since $J_k$ is locally constant in the normal direction near
$L_k$, we may define \begin{equation} \widetilde{J}_0( {\bf q},
0,\Theta) = \widetilde{J}_0( {\bf q}, \epsilon_1/2,\Theta ), \quad
 \widetilde{J}_1( {\bf q}, 1,\Theta) := \widetilde{J}({\bf q},
  1 - \epsilon_1/2, \Theta ). \end{equation}
 The space $ GL^{+}(2n,\RR)/GL(n,\CC)  $ is path connected. Hence there exists a smooth
  path \begin{equation} F(t): [0.4, 0.6] \to
 GL^{+}(2n,\RR)/GL(n,\CC), \; F(0.4) = \widetilde{J}_1 ( {\bf 1}, 1, \Theta ),
 \; F(0.6) = \widetilde{J}_0 ( {\bf 1}, 0, \Theta ).
  \end{equation}
  By $\TT^n$-invariance, we construct a smooth family of paths
 $ F( {\bf q}, t ):  \TT^n   \times  [0.4, 0.6] \to
 GL^{+}(2n,\RR)/GL(n,\CC)$,
 \begin{equation}  F({\bf q}, t) := d{\bf q} F(t) d{\bf q}^{-1},
 \end{equation}
 satisfying $ F({\bf q}, 0.4) = \widetilde{J}_1 ( {\bf q},  1, \Theta  ),
 \; F({\bf q}, 0.6) = \widetilde{J}_0 ({\bf q}, 0, \Theta ).$

\noindent Choose a smooth non-decreasing function $\alpha: (0,1)
\to [0,1) $ such that
\begin{equation} \alpha(t) = \left\{ \begin{array}{ll}
                              t & {\rm if} \; t \geq 0.8 \\
                              0 & {\rm if} \; t \leq 0.6.
                              \end{array} \right.
\end{equation}
Choose another smooth non-decreasing function
  $\beta: (0,1) \to (0,1] $ such
that
\begin{equation} \beta(t) = \left\{ \begin{array}{ll}
                              t & {\rm if} \; t \leq 0.2 \\
                              1 & {\rm if} \; t \geq 0.4.
                              \end{array} \right.
\end{equation}
Define a map $\widetilde{J}: \TT^n   \times B^{-}  \to
GL^{+}(2n,\RR)/GL(n,\CC) $ by
\begin{equation} \widetilde{J} ( {\bf q}, r, \Theta) = \left\{ \begin{array}{ll}
                              \widetilde{J}_0 ( {\bf q}, \alpha(r), \Theta ) & {\rm if} \; r> 0.6 \\
                              F( {\bf q}, r ) & 0.6 \geq r \geq 0.4 \\
                               \widetilde{J}_1 ( {\bf q}, \beta(r), \Theta )  & {\rm if} \; r< 0.4.
                              \end{array} \right.
\end{equation}
Note that
\begin{equation} \widetilde{J} ({\bf q}, r, \Theta) = \left\{ \begin{array}{ll}
                              \widetilde{J}_0 ( {\bf q}, r, \Theta ) & {\rm if} \; r> 0.8 \\
                               \widetilde{J}_1 ( {\bf q}, r, \Theta )  & {\rm if} \; r<
                               0.2.
                              \end{array} \right.
\end{equation}
Define a $\TT^n$-invariant almost complex structure $\bar{J}_k$ on
 $T(U_k -L_k) $ by
\begin{equation} \bar{J}_k  =  d\phi_k^{-1} \circ \widetilde{J} \circ  d\phi_k.
\end{equation}
By construction, $\bar{J}_k$ agrees with $J_k$ in a neighborhood
of the outer boundary of $U_k - L_k$. Therefore $\bar{J}_k$
extends to a $\TT^{n}$-invariant almost complex structure on $M_k
- L_k$. Moreover $\bar{J}_0 \circ dh = dh \circ \bar{J}_1 $ on
$U_1 - L_1$ since $h = \phi_0^{-1} \circ \phi_1$, see \eqref{h}
and \eqref{phi}. Therefore $\bar{J}_0$ and $\bar{J}_1$ glue to
produce a $\TT^{n}$-invariant almost complex structure $\bar{J}$
on $M$. Finally, note that we may apply induction when the number
of holes is greater than one.
\end{proof}

\begin{remark}\label{gkacs} In dimension four, the sufficiency
part of Theorem \ref{acs} also follows from section 13 of
\cite{GK} together with the main theorem of \cite{Kus}.
\end{remark}

\subsection{The $\chi_y$ genus} The Hirzebruch $\chi_y$ genus is an
invariant of the complex cobordism class of the manifold and thus
depends on the stable almost complex structure. We give a
combinatorial formula of the $\chi_y$ genus of $M$, following
Panov's work on quasitoric manifolds. The proofs are the same as
in \cite{Pan}.

 Let $E$ be an edge of $P^n$. The
isotropy subgroup of $\pi^{-1}(E)$ is an $(n-1)$-dimensional torus
generated by a  submodule $K$ of rank $(n-1)$ in $\ZZ^n$. A
primitive vector $\mu$ in $(\ZZ^n)^{\ast}$ is called an edge
vector corresponding to $E$ if $\mu(\alpha)= 0 $ for each $\alpha
\in K$. The edge vector of $E$ is therefore unique up to sign.

Let $\nu$ be a primitive vector in $\ZZ^n$ such that
\begin{equation}\label{nu} \mu (\nu) \neq 0 \; {\rm for\; any \;
edge\; vector\; \mu}. \end{equation}
 Then the circle $ S_{\nu}^1 = (\ZZ<\nu> \otimes \RR)/\ZZ<\nu>  $
  acts smoothly on $M$ with only isolated
fixed points corresponding to the vertices of $P$.

We choose signs for each edge vector at a vertex $v$ according to
the characteristic omniorientation as follows. Order the
codimension one faces meeting at $v$ and corresponding edges
$E_k$s as in subsection \ref{omni}. Let $\mu_k$ be an edge vector
corresponding to $E_k$. Let $M_{(v)}$ be the matrix, $M_{(v)}=
[\mu_1, \ldots, \mu_k] $. Then choose sign for each $\mu_k$ such
that $M_{(v)}^t \Lambda_{(v)}= I_{n} $. Under this choice of signs
the action of $S^1_{\nu}$ induces a  representation of $S^1$ on
the tangent space $T_v M$ with weights $\mu_1(\nu), \ldots,
\mu_n(\nu)$.

\begin{defn}   Define the index
of a vertex $v \in P$ as the number of negative weights of the
$S^1$ representation on $T_v(M)$,
$$ {\rm ind}_{\nu}(v) = |\{k: \mu_k(\nu) < 0 \} |. $$
\end{defn}

\begin{theorem}\label{chiy} For any vector $\nu$ satisfying \eqref{nu},
$$ \chi_y (M ) = \sum_v (-y )^{{\rm ind}_{\nu}(v)} \sigma(v). $$
\end{theorem}

Note that the Theorem \ref{chiy} also follows from
% the Kosniowski formula or
Lemma \ref{stable} together with Theorem 10.1 of \cite{HM}.
Specializing the formula in Theorem \ref{chiy} to $y= -1$ and $y=
1 $, respectively yield formulas for the top Chern number and the
signature. Moreover following Theorem 3.4 of \cite{Pan} or Theorem
4.2 of \cite{Mas} we obtain the following formula for Todd genus
of $M$,
\begin{equation}\label{todd}
 {\rm td}(M) = \sum_{{\rm ind}_{\nu}(v)=0} \sigma(v).
\end{equation}

\subsection{Integrability questions}\label{iq}

\begin{lemma}\label{symp1} If the polytope $P$ has at least one
hole,
then the torus manifold $M(P, \lambda)$ does not support any
symplectic form for which the torus action is symplectic.
\end{lemma}

\begin{proof}
When the dimension $2n > 4$,  $M(P, \lambda)$ is simply connected.
So any symplectic circle action is Hamiltonian. Therefore if $M(P,
\lambda)$ supports a $T^n$-invariant symplectic form, then the
action of $T^n$ must be Hamiltonian. Then $M(P, \lambda)$ would be
a symplectic toric manifold  with a moment map whose image is a
Delzant polytope. Then the orbit space of the $T^n$-action on
$M(P, \lambda)$ would be a Delzant polytope, see Theorem 2.6.2 of
\cite{Sil}. Therefore, as the orbit space of $M(P, \lambda)$ is
not convex it cannot support an invariant symplectic form.

When $2n=4$, a result of McDuff \cite{McD} states that a
symplectic circle action on a compact four dimensional manifold is
Hamiltonian if and only if it has fixed points. Therefore, again,
 if $M(P,\lambda)$ supports a $T^n$-invariant symplectic form, then the
action of $T^n$ must be Hamiltonian. We get a contradiction as
above.
\end{proof}

 It follows from the main result of \cite{IK} that $M(P,
\lambda)$ cannot admit a complex structure with respect to which
the torus action is holomorphic if it is not a toric variety, for
instance when $P$ has at least one hole.

More generally, we may ask whether $M(P,\lambda)$ admits any
symplectic or complex structure. We do not know of any example
that does so in case $P$ has at least one hole.

\begin{lemma}\label{kahler} If $P$ is a $2$-polytope with an odd number of holes, 
then $M(P,\lambda)$ can not be Kahler.
\end{lemma}

\begin{proof}
 If $P$ has $s$
holes, by Corollary \ref{homt5}, the first Betti number $b_1(M)= s$.
But for a compact Kahler manifold, the Betti numbers of odd degree are 
even (see \cite{GH}, page 117). The result follows.
\end{proof}

 It is not hard to produce
examples of $M$ that admit almost complex structure but do not
admit an integrable complex structure. For an almost complex
$4$-manifold, $c_1^2$ and $c_2$ are determined by the Euler
characteristic and signature, and are therefore independent of the
choice of almost complex structure. Consider the equivariant
connected sum $Y$ of three copies of $\CC P^2$. This is a
quasitoric manifold with a pentagon as $P$. The characteristic
vectors may be chosen to be $(1,0)$, $(-1,1 )$, $(1,-2)$, $(0,1)$
and $(-1,-1)$, thus endowing $Y$ with a positive omniorientation
and an almost complex structure. However, $Y$ has $c_1^2 = 19 $
and $c_2=5$. Therefore the Bogomolov-Miyaoka-Yau inequality,
$c_1^2 \le 3c_2$, is not satisfied and $Y$ does not admit a
complex structure. It may be argued using \eqref{chern} and
intersection theory that $c_1^2$ and $c_2$ are additive with
respect to the fiber sum operation. Therefore the fiber sum of any
finite number of copies of $Y$ produces an almost complex manifold
which does not admit a complex structure. Since $c_1^2= 3c_2$ for
$\CC P^2$, the fiber sum of copies of $Y$ and $\CC P^2$ also
yields such examples.

%\subsection{Examples}

%Consider the manifolds obtained by symplectic fiber sum of four
%dimensional toric manifolds.   Hence all of them obey the
%Bogomolov-Miyaoka-Yau inequality: $c_1^2 \le 3c_2$.

%So we have to
%use further details from the Enrique-Kodaira classification of
%surfaces (see \cite{BPV}) to exhibit one that is not complex.

%\begin{figure}[ht]
 %       \centerline{
  %         \scalebox{0.70}{
  %          \input{ega4.pstex_t}
   %         }
    %      }
  % \caption {Some symplectic but non-complex torus manifolds}
   %\label{figeg04}
   %\end{figure}

% For example, suppose $M$ is obtained by fiber summing $a$
%copies of ${\bf P}^2$ with $b$ many Hirzebruch surfaces. A
%Hirzebruch surface has characteristic vectors
%$(1,0),\,(0,1),\,(-1,k)$ and $(0,-1)$, in that order. One may
%verify using remark \ref{int} that any such complex surface has
%$c_1^2 = 8$, whereas $c_2 = 4$. Therefore $c_1(M)^2 = 9a + 8b$ and
%$c_2(M) = 3a + 4b$.

%Now consider the case when there is only one hole, i.e. $a+ b =2$.
%Then by Theorem \ref{homt4}, the first Betti number $b_1(M)= 1$.
%As Gompf points out in p. 560 of \cite{Gom}, it follows from the
%the Enrique-Kodaira classification that no $4$-dimensional
%symplectic manifold with  $b_1=1$ can be homotopy equivalent to a
%complex surface.
 %Therefore none of these manifolds are complex if
%$a+ b =2$.
\bigskip

{\bf ACKNOWLEDGEMENT.} It is a pleasure to thank Shengda Hu for
extensive discussions. We also thank Andres Angel, Saibal Ganguli,
Mikhail Malakhaltsev, Taras Panov and Dong Youp Suh for helpful
conversations. We thank Mikiya Masuda for comments that helped us
improve our exposition. We thank Yael Karshon
 and the referee for pointing out two different serious
errors in earlier drafts of the article. We also thank the referee for
 suggesting numerous improvements. The first
author was partially supported by the Proyecto de investigaciones
grant from Universidad de los Andes. The second author was
partially supported by the National Research Foundation of Korea
(NRF) grant funded by the Korea government (MEST) (No.
2012-0000795).

\renewcommand{\refname}{References}

\vspace{1cm}

\vfill

\end{document}